\documentclass[12pt, a4paper]{article}
\usepackage{amssymb,amsfonts,amsthm,amsmath}
\usepackage{caption}
\usepackage{float}
\usepackage{setspace}
\usepackage{epic}
\usepackage{graphics}
\usepackage{graphicx}
\usepackage[T1]{fontenc}
\usepackage{enumerate}
\usepackage{indentfirst}
\usepackage[margin=2.9cm]{geometry}
\allowdisplaybreaks

\newcommand{\floorfrac}[2]{\left\lfloor \frac{#1}{#2} \right\rfloor}
\newcommand{\KL}{Kazhdan-Lusztig\ }
\newtheorem{thm}{Theorem}[section]
\newtheorem{lem}[thm]{Lemma}
\newtheorem{prop}[thm]{Proposition}
\newtheorem{cor}[thm]{Corollary}
\newtheorem{conj}[thm]{Conjecture}

\newtheorem{defn}{Definition}[section]
\numberwithin{equation}{section}
\DeclareMathOperator{\B}{\mathcal{B}}
\usepackage{mma}
\usepackage{ytableau}
\usepackage[colorlinks,
linkcolor=blue,
anchorcolor=blue,
citecolor=blue
]{hyperref}
\begin{document}
\begin{center}
{\large \bf  Log-concavity of inverse Kazhdan-Lusztig polynomials of paving matroids}
\end{center}
\begin{center}
 Matthew H.Y. Xie$^1$ and Philip B. Zhang$^{2}$\\[6pt]
 $^{1}$ College of Science \\
Tianjin University of Technology, Tianjin 300384, P. R. China\\[6pt]

$^{2}$ College of Mathematical Science \& Institute of Mathematics and Interdisciplinary Sciences,
Tianjin Normal University, Tianjin  300387, P. R. China\\[6pt]
Email: 	 $^{1}${\tt xie@email.tjut.edu.cn},
 $^{2}${\tt zhang@tjnu.edu.cn}
\end{center}
\noindent
\textbf{Abstract.}
Gao and Xie (2021) conjectured that the inverse Kazhdan-Lusztig polynomial of any matroid is log-concave.
Although the inverse Kazhdan-Lusztig polynomial may not always have only real roots, we conjecture that the
Hadamard product of  an inverse Kazhdan-Lusztig polynomial of degree $n$ and $(1+t)^n$ has only real roots.
Using interlacing polynomials and multiplier sequences, we confirm this conjecture for paving matroids. This result allows us to confirm the log-concavity conjecture for these matroids by applying Newton's inequalities.

\noindent\textbf{AMS Classification 2010:} 05A20, 05B35, 33F10, 26C10

\section{Introduction}
The goal of this paper is to prove the log-concavity of the inverse Kazhdan-Lusztig polynomials of paving matroids.
Elias, Proudfoot, and Wakefield~\cite{elias2016kazhdan} introduced the notion of the \KL polynomial of a matroid.
Similarly, Gao and Xie \cite{gao2021inverse} introduced the inverse Kazhdan-Lusztig polynomial of a matroid.
Recently, Braden, Huh, Matherne, Proudfoot, and Wang~\cite{braden2020singular} proved that the \KL polynomials and the inverse \KL polynomials of matroids have only non-negative coefficients.
However, the log-concavity of \KL polynomials and inverse \KL polynomials of matroids is still open and only a few results have been obtained, see \cite{xie2023log,wu2023log}.
A sequence $ a_0, a_1, a_2, \ldots, a_n$  is said to be \textit{log-concave and has no internal zeros} if it satisfies $ a_i^2 \ge a_{i-1}a_{i+1}$ for any $ 1 \le i \le n-1 $, and the indices of its nonzero elements form consecutive integers. A polynomial is said to be log-concave if its coefficients form a log-concave sequence. Gao and Xie~\cite{gao2021inverse}  conjectured that the coefficients of the inverse Kazhdan-Lusztig polynomials of matroids are log-concave and have no internal zeros. We use $Q_M(t)$ to denote the inverse Kazhdan-Lusztig polynomial of the matroid $M$.
\begin{conj}[\cite{gao2021inverse}]\label{LC}
For any matroid $M$, the coefficients of $Q_M(t)$ form a log-concave sequence and have no internal zeros.
\end{conj}
Elias, Proudfoot, and Wakefield~\cite{elias2016kazhdan} conjectured that the Kazhdan-Lusztig polynomials have only real roots. However, Gao and Xie pointed out that inverse Kazhdan-Lusztig polynomials do not necessarily have only real roots. Nevertheless, we conjecture that a variation of inverse Kazhdan-Lusztig polynomials has only real roots.
Recall that a matroid is called \emph{paving} if the size of every circuit is at least the rank of $M$.
The main object of this paper is as follows.
\begin{thm}\label{paving-log}
For any paving matroid $M$, the coefficients of $Q_M(t)$ form a log-concave sequence and have no internal zeros.
\end{thm}
We prove Theorem \ref{paving-log} via the approach of real-rooted polynomials.
Given a polynomial of degree $n$
\[
p(t)=a_0+a_1 t+\cdots+ a_n t^n,
\]
we denote
\[
\B(p(t))= \binom{n}{0} a_0+\binom{n}{1} a_1 t +\cdots+ \binom{n}{n}a_n t^n,
\]
which is the Hadamard product of $(1+t)^n$ and $p(t)$.
Note that the operator $\B$ is not linear for two polynomials with distinct degrees.
\begin{thm}\label{paving-rz}
For any paving matroid $M$, the polynomial $\B(Q_M(t))$ has only real roots.
\end{thm}
By Newton's inequalities, Theorem \ref{paving-rz} implies Theorem \ref{paving-log}.
In general, we make the following conjecture, which would imply Conjecture \ref{LC}.
\begin{conj}\label{TRZ}
For any simple matroid $M$, the polynomial $\B(Q_M(t))$ has only real roots.
\end{conj}
Recall that a sequence satisfying the Turán inequality is also called a log-concave sequence.
The higher order Turán inequalities are defined by
$4(a_{i}^2-a_{i-1}a_{i+1})(a_{i+1}^2-a_{i}a_{i+2})-(a_{i}a_{i+1}-a_{i-1}a_{i+2})^2\geq 0$.
In recent years, the higher-order Turán inequalities for combinatorial sequences have attracted significant attention, as seen in \cite{chen2019higher,ono2022turan}.
A theorem of Ma\v{r}\'{\i}k~\cite{marik1964} states that if a real polynomial
$\binom{n}{0} a_0+\binom{n}{1} a_1 t+\cdots+ \binom{n}{n}a_n t^n$ of degree $n \geq 3$ has only real roots, then $a_0, a_1,\ldots, a_n$ satisfies the higher-order Turán inequalities.
By Theorem \ref{paving-rz},  we get the following result.
\begin{cor}
For any paving matroid $M$, the coefficients of $Q_M (t)$ satisfy the higher order Turán inequalities.
\end{cor}
The structure of this paper is as follows.
In Section 2, we provide the necessary background on real-rooted polynomials.
Section 3 focuses on the log-concavity of the inverse Kazhdan-Lusztig polynomials of paving matroids, presenting the main proofs and results.
Finally, in Section 4, we propose several conjectures regarding the degree of $Q_M(t)$.
\section{Interlacing and multiplier sequence}
In this section, we recall some useful results of interlacing polynomials and multiplier sequences.
Given two real-rooted polynomials $f(t)$ and $g(t)$ with positive leading coefficients, let $\{u_i\}$ be the set of roots of $f(t)$ and let $\{v_j\}$ be the set of roots of $g(t)$. We say that $g(t)$ interlaces $f(t)$, denoted $g(t) \ll f(t)$, if either of the following two conditions is satisfied:
\begin{itemize}
\item[(1)] $\deg f(t)=\deg g(t)=n$ and
\begin{align*}
v_n\le u_n\le v_{n-1}\le\cdots\le v_2\le u_2\le v_1\le u_1;
\end{align*}
\item[(2)]
$\deg f(t)=\deg g(t)+1=n$ and
\begin{align*}
u_{n}\le v_{n-1}\le\cdots\le v_{2}\le u_{2}\le v_{1}\le u_{1}.
\end{align*}
\end{itemize}
For convenience, we let $0 \ll f$ and $f \ll 0$ for any real-rooted polynomial $f$.
Following Br{\"a}nd{\'e}n~\cite{Braenden2015Unimodality},
a sequence of real polynomials $(f_{1},\dots,f_{m})$
with positive leading coefficients is said to be \emph{interlacing} if $f_{i} \ll f_{j}$ for all $1\le i<j\le m$.
The following lemma plays an important role in our paper.
\begin{lem}[\cite{wagner1992total}, Proposition 3.3]\label{WeakTransitivity}
If $f_0,f_1, f_2, \ldots, f_{m}$ are real-rooted polynomials with $f_0 \ll f_{m}$ and
$f_{i-1} \ll f_i$ for all $1 \leq i \leq m$, then the polynomial sequence $(f_0, f_1, \ldots, f_m)$ is interlacing.
\end{lem}
We also need the following lemma.
\begin{lem}[\cite{wagner1992total}, Proposition 3.5]\label{lem:convex}
	Let $f$, $g$, and $h$ be real-rooted polynomials  with nonnegative coefficients.
	\begin{itemize}
		\item  	If $f \ll g$ and $f \ll h$, then $f \ll g+h$.
		\item 	If $f \ll g$ and $h \ll g$, then $f+h \ll g$.
	\end{itemize}
\end{lem}
We also say that the roots of $g(t)$ and $f(t)$ are alternating if $f(t) \ll g(t)$ or $g(t) \ll f(t)$.
The main tool we use to prove the property of alternating roots is the following Hermite-Kakeya-Obreschkoff theorem.
\begin{thm}[Hermite-Kakeya-Obreschkoff, {\cite[Theorem~2.4]{branden2006linear}, \cite[Theorem~4.1]{dedieu1992obreschkoff}}]
\label{HKO}
Let $p(t)$ and $q(t)$ be two polynomials with real coefficients. The roots of $p(t)$ and $q(t)$ are
 alternating if and only if for any real numbers $ \lambda$ and $\mu$, the polynomial $\lambda p(t)+\mu q(t)$ has only real roots.
\end{thm}
However, in this form of the Hermite-Kakeya-Obreschkoff theorem, we cannot determine whether $g(t)$ interlaces $f(t)$ or $f(t)$ interlaces $g(t)$. To address this, we next recall the Wro\'nskian notation of two polynomials.
The Wro\'nskian of two polynomials $f(t)$ and $g(t)$ is defined as $W[f, g] = f' g - f g'$. It is easy to see that for any pair of polynomials $f(t)$ and $g(t)$ with interlacing roots, the Wro\'nskian $W[f, g]$ is either nonnegative on all of $\mathbb{R}$ or nonpositive on all of $\mathbb{R}$. Moreover, if $W[f, g] \leq 0$, then $f \ll g$ (see \cite[P. 4]{wagner2011multivariate}). Therefore, we only need to check the signs of $W[f, g]$ for certain values of $t$.
`
\begin{defn}
A sequence $\Gamma=\{\gamma_k\}_{k=0}^{\infty}$ of real numbers is called a \emph{multiplier sequence} if, whenever any real polynomial
$f(t)=\sum_{k=0}^{n}a_kt^k$ has only real roots, so does the polynomial $\Gamma[f(t)]=\sum_{k=0}^{n}\gamma_ka_kt^k.$
\end{defn}
\begin{defn}
A sequence $\Gamma=\{\gamma_i\}_{k=0}^{n}$ is called an \emph{$n$-sequence} if for every polynomial $f(t)$ of degree less than or equal to $n$ and with only real roots, the polynomial $\Gamma[f(t)]$ also has only real roots.
\end{defn}
\begin{thm}[{P\'olya-Schur, \cite{polya1914zwei}, \cite[Theorem 3.3]{craven2004composition}}]\label{po}
	Let $\Gamma=\{\gamma_k\}_{k=0}^{\infty}$ be a sequence of real numbers. Then $\Gamma$ is a multiplier sequence if and only if the roots of the polynomial 	$\Gamma[(1+t)^n]$ are all real and of the same sign for all integers $n>0$.
	\end{thm}
\begin{thm}[\cite{craven1983location1}, Algebraic Characterization of $n$-sequences]\label{ch-n}
Let $\Gamma=\{\gamma_k\}_{k=0}^{n}$ be a sequence of real numbers. Then $\Gamma$ is an $n$-sequence if and only if the roots of the polynomial
$\Gamma[(1+t)^n]$ are all real and of the same sign.
\end{thm}
The following two lemmas will be used frequently later. Although they  are well-known and the proofs are straightforward, we provide them here for  the reader’s convenience
\begin{lem}\label{ms1}
 For positive integers $k\geq l\geq 1$, we have
 \[
 \left\{ \frac{1}{i! \prod_{j=l}^{k} {(i+j)} } \right\}_{i=0}^{\infty}
 \]
 is a multiplier sequence.
\end{lem}
\begin{proof}
It is clear that
 \begin{align*}
 \frac{1}{i! \prod_{j=l}^{k} {(i+j)} } = \frac{\prod_{j=l}^{k} {(i+j)} }{(i+k)! }.
 \end{align*}
Applying a classical theorem due to Laguerre to the function $\frac{1}{\Gamma(t+1+k)}$ and $\prod_{j=l}^{k} {(t+j)} $, it is easy to see that both $ \frac{1 }{(i+k)! }$ and $\prod_{j=l}^{k} {(i+j)} $ are multiplier sequences, see also \cite[p.~270]{titchmarsh1939theory} or a more recent literature \cite[Theorem 4.1]{craven2004composition}.
\end{proof}
\begin{lem}\label{ns1}
 For integers $k\geq l\geq1$, we have
 \[
 \left\{ \frac{1}{(d-i)! \prod_{j=l}^{k} {(d-i+j)} } \right\}_{i=0}^{d}
 \]
 is a $d$-sequence.
 In particular, for any nonnegative integer $m$, \[
	\left\{ \frac{1}{(d-i+m)! } \right\}_{i=0}^{d}
	\]
	is a $d$-sequence.
\end{lem}
\begin{proof}
By Theorem \ref{ch-n}, we need to show that
\begin{align}
\sum_{i=0}^{d} \binom{d}{i} \frac{1}{(d-i)! \prod_{j=l}^{k} {(d-i+j)} } t^i =\sum_{i=0}^{d} \binom{d}{i} \frac{1}{i! \prod_{j=l}^{k} {(i+j)} } t^{d-i}\label{ns2}
\end{align}
has only real roots.
The real-rootedness of the right-hand side of \eqref{ns2} is equivalent to that of $$\sum_{i=0}^{d} \binom{d}{i} \frac{1}{i! \prod_{j=l}^{k} {(i+j)} } t^{i},$$ which follows from Lemma \ref{ms1}.
 \end{proof}
In this paper, we often need to prove that certain specific cubic or quadratic polynomials have only real roots. Although the form is consistent, the degree may vary. We require the following fact.
\begin{lem}\label{disc}
If $$\Delta(a,b,c):=18\,abcd-4\,b^{3}d+b^{2}c^{2}-4\,ac^{3}-27\,a^{2}d^{2}>0,$$
then
the polynomial $f(t):=a t^{3}+b t^{2}+c t+d$ has only real roots.
\end{lem}
\begin{proof}
If $a\not = 0$.
Recall that $\Delta(a,b,c)$ is just the discriminant of $f(t)$.
Then the cubic polynomial $f(t)$ has three distinct real roots by the definition of the discriminant.
If $a=0$, we have $\Delta(0,b,c)=b^2(c^2-4 b d )>0$. Thus, $b\not=0$ and $f(t)$ have two distinct real roots.
\end{proof}
\section{Log-concavity for paving matroid}
This section aims to prove Theorem \ref{paving-rz}. Our approach relies on the alternating of roots.
Recall that Gao and Xie~\cite{gao2021inverse} obtained the following explicit formula for the inverse \KL polynomials of any uniform matroid,
\begin{align}\label{Qmd}
Q_{U_{m,d}}(t)=\sum_{i=0}^{\lfloor (d-1) /2\rfloor}\frac{m(d-2i)}{(m+i)(d-i+m)}\binom{m+d}{d}\binom{d}{i}t^i.
\end{align}
Ferroni, Nasr, and Vecchi~\cite{ferroni2023stressed} obtained the inverse \KL polynomial of any paving matroid.
\begin{thm}[{\cite[Theorem 4.4]{ferroni2023stressed}}]\label{paving_coeff}
 Let $M$ be a paving matroid of rank $d$ and cardinality $m+d$. Suppose $M$ has exactly $\lambda_h$ (stressed) hyperplanes of cardinality $h+d-1$. Then
 \begin{align*}
 Q_{M}(t)=Q_{U_{m,d}}(t)-\sum_{h=1}^{m} \lambda_h\ (Q_{U_{h,d}}(t)-Q_{U_{h,d-1}}(t)).
 \end{align*}
\end{thm}
Instead of directly proving Theorem \ref{paving-rz}, we prove
$$\B(Q_{U_{m,d}}(t))- \sum_{h=1}^{m} \lambda_h \ \B\left( (Q_{U_{h,d}}(t)-Q_{U_{h,d-1}}(t)) \right) $$
 has only real roots for any nonnegative numbers $\lambda_h$. To establish this, we will prove the following lemma.
\begin{lem}\label{h_m_int}
For any $1\leq h\leq m$, we have that
$$\B(Q_{U_{m,d}}(t)) \ll \B(Q_{U_{h,d}}(t)-Q_{U_{h,d-1}}(t)).$$
\end{lem}
Due to the computational difficulty, we will prove the following lemmas instead of proving Lemma \ref{h_m_int}.
\begin{lem}\label{m1}
For any positive integer $m$ and $d\geq 3$ , we have
$$\B(Q_{U_{m,d}}(t)) \ll \B(Q_{U_{1,d}}(t)-Q_{U_{1,d-1}}(t)).$$
\end{lem}
\begin{proof}
 We first prove that $\B(Q_{U_{m,d}}(t))$ has only real roots.
From \eqref{Qmd}, we have
 \begin{align*}
 \B(Q_{U_{m,d}}(t)) =\frac{(m+d)!}{(m-1)!}\sum _{i=0}^{\floorfrac{d-1}{2} }
 \frac{(d-2i)}{i!(i+m)(d-i)!(d-i+m)} {\floorfrac{d-1}{2} \choose i } t^i.
 \end{align*}
Thus, the polynomial \begin{align*}
	\sum _{i=0}^{\floorfrac{d-1}{2} }
 (d-2i) \binom{\floorfrac{d-1}{2}}{i} t^i
 &=\sum _{i=0}^{\floorfrac{d-1}{2} }d \binom{\floorfrac{d-1}{2}}{i} t^i
 -\sum _{i=0}^{\floorfrac{d-1}{2} } 2i \binom{\floorfrac{d-1}{2}}{i} t^i\\
 &=d(1+t)^{\floorfrac{d-1}{2} }-2 t {\floorfrac{d-1}{2} }( 1+ t)^{\left\lfloor\frac {d - 1} {2} \right \rfloor - 1} \\
 &=(1+t)^{\left\lfloor\frac {d - 1} {2} \right \rfloor - 1} \left(d + d t -2t\left\lfloor\frac {d - 1} {2} \right \rfloor \right)
 \end{align*}
 has only negative roots.
By Lemmas \ref{ms1} and \ref{ns1} we obtain the desired result.
By the Hermite-Kakeya-Obreschkoff theorem, it remains to prove that the following polynomial has only real roots for any $k$.
By direct computation, we have
 \begin{align*}
 &\frac{k}{m \binom {d + m} {d}}\B(Q_{U_{m,d}}(t))+ \B(Q_{U_{1,d}}(t)-Q_{U_{1,d-1}}(t)) \\
 &=\sum _{i=0}^{\floorfrac{d-1}{2} }
 \left( k \frac{(d-2i)}{(m+i)(d-i+m)} + \frac {d - 2 i + 1} {d - i + 1} \right) \binom {d} {i} {\floorfrac{d-1}{2} \choose i } t^i
 \\
 &=d! \sum _{i=0}^{\floorfrac{d-1}{2} }
 \frac {k (d-2 i) (d-i+1)+ (d-2 i+1) (i+m) (d-i+m) } {i! (i + m) (d - i + 1)! (d - i + m)} {\floorfrac{d-1}{2} \choose i } t^i
 \end{align*}
 For any real number $k$ and positive integers $d\geq 7$ and $m\geq 2$, we will show that
 \[
 \sum _{j=0}^{\floorfrac{d-1}{2} } \left( k (d-2 i) (d-i+1)+ (d-2 i+1) (i+m) (d-i+m)\right) { \floorfrac{d-1}{2} \choose i }t^i .
 \]
 has only real roots.
 Let $n=\floorfrac{d-1}{2}$, then $n \geq 3$.
 It is clear that $d=2n+1$ if $d$ is odd and $d=2n+2$ if $d$ is even.
 By direct calculation, we have
 \begin{align*}
 &\sum_{i=0}^{n}
 \left( k (d-2 i)(d-i+1)+ (d-2 i+1) (i+m) (d-i+m) \right) {n \choose i } t^i =(t+1)^{n-3} p_1(t),
 \end{align*}
 where
 $ p_1(t)= d k + d^2 k + d m + d^2 m + m^2 +
 d m^2 + (3 d k + 3 d^2 k + 3 d m + 3 d^2 m + 3 m^2 + 3 d m^2 + n -
 2 d n + d^2 n - 3 d k n - 2 d m n - 2 m^2 n) t + (3 d k +
 3 d^2 k + 3 d m + 3 d^2 m + 3 m^2 + 3 d m^2 - 3 n - d n +
 2 d^2 n - 2 k n - 6 d k n - 4 d m n - 4 m^2 n + 5 n^2 - 3 d n^2 +
 2 k n^2) t^2 + (d k + d^2 k + d m + d^2 m + m^2 + d m^2 + d n +
 d^2 n - 2 k n - 3 d k n - 2 d m n - 2 m^2 n - n^2 - 3 d n^2 +
 2 k n^2 + 2 n^3) t^3.$
 Now we prove that the polynomial $p_1(t)$ has only real roots for both odd and even values of $d$ using \textbf{Mathematica}.
 \begin{mma}
 \In p1= |Collect|\left[ {\sum _ {i = 0}^
 n t^i\binom {n} {i} ( k (d - i + 1) (d - 2 i) + (d - 2 i + 1) (i + m) (d - i +
 m) )}\right./ (t + 1)^{n - 3},
 t, \left. \right];\\
 \In disc = |Discriminant|[p1, t];\\
 \In |Resolve|[|ForAll|[\{k, n, m, d\}, n\geq3 \&\& m \geq 2 \&\& (d == 2n + 1\|d==2n+2), disc > 0]]\\
 \Out True\\
 \end{mma}
 Using Lemma \ref{disc}, we know that $p(t)$ has only real roots. By using Lemmas~\ref{ms1} and~\ref{ns1} we can obtain the desired result.
 For the case $m=1$ and $d\geq 7$, we have
 \begin{align*}
 &\frac{k}{ \binom {d + 1} {d}}\B(Q_{U_{1,d}}(t))+ \B(Q_{U_{1,d}}(t)-Q_{U_{1,d-1}}(t)) \\
 &=d! \sum _{i=0}^{\floorfrac{d-1}{2} }
 \frac {k (d - 2 i) + (i + 1) (d - 2 i + 1)} {(i + 1)! (d - i + 1)!} {\floorfrac{d-1}{2} \choose i } t^i
 \end{align*}
 Similar to the case $m\geq2$, we only need to show that
 \begin{align*}
	&\sum_{i=0}^{n}
	\left( k (d-2 i)+ (d-2 i+1) (i+1) \right) {n \choose i } t^i =(t+1)^{n-2} p_2(t),
	\end{align*}
where $p_2(t)=1 + d + d k + (2 + 2 d + 2 d k - 3 n + d n - 2 k n) t + (1 + d + d k -	n + d n - 2 k n - 2 n^2) t^2$.
The discriminant of $p_2(t)$ is
 $ ((d n - 2 d + 3 n) + 2 k n)^2 + 4 (n - 1) (d (d - n) + 2 n),$
 which is obviously positive. Therefore, $p_2(t)$ is of degree 1 or 2, and in both cases, it has only real roots.
 If $d = 3$ or $d = 4$, the degrees of $\lambda \B(Q_{U_{m,d}}(t))$ and $\B(Q_{U_{1,d}}(t) - Q_{U_{1,d-1}}(t))$ are both 1.
For $d = 5$ or $d = 6$, their degrees are both 2. We can verify the desired result directly using Mathematica .
We have proven that for any real numbers $ \lambda$ and $\mu$, the polynomial $\lambda \B(Q_{U_{m,d}}(t)) +\mu \B(Q_{U_{1,d}}(t)-Q_{U_{1,d-1}}(t)) $ has only real roots. Thus the roots of $\B(Q_{U_{m,d}}(t))$ and $\B(Q_{U_{1,d}}(t)-Q_{U_{1,d-1}}(t)) $ are alternating. To finish the proof, we only need to show that
$$W[\B(Q_{U_{m,d}}(t)),\B(Q_{U_{1,d}}(t)-Q_{U_{1,d-1}}(t))]$$ is negative at $t=0$.
Recall that
$$Q_{U_{m,d}}(t)=\binom {d + m - 1} {m}+ \frac {(d - 2) m} {d + m - 1}\binom {d + m} {m + 1} t+\cdots.$$
Thus
\begin{align*}
&W[\B(Q_{U_{m,d}}(t)),\B(Q_{U_{1,d}}(t)-Q_{U_{1,d-1}}(t))]\\
&=(\B(Q_{U_{m,d}}(t)))^{'}\B(Q_{U_{1,d}}(t)-Q_{U_{1,d-1}}(t)) - \B(Q_{U_{m,d}}(t)) \B(Q_{U_{1,d}}(t)-Q_{U_{1,d-1}}(t))^{'}
\end{align*}
The constant term is
\begin{align*}
	&\binom{\floorfrac{d-1}{2}}{1}\frac {(d - 2) m} {d + m - 1}\binom {d + m} {m + 1}\cdot 1 - \binom{\floorfrac{d-1}{2}}{1} \binom {d + m - 1} {m} \cdot (d-1) \\
	&=-\frac {\left(d m + (d - 1)^2 + m^2 \right) (d + m - 2)!} {(d - 1)! (m + 1)!},
\end{align*}
which is negative.
We complete the proof.
 \end{proof}
\begin{lem}\label{h+1h}
For any positive integer $h$, we have
$ \B(Q_{U_{h+1,d}}(t)-Q_{U_{h+1,d-1}}(t)) \ll \B(Q_{U_{h,d}}(t)-Q_{U_{h,d-1}}(t)) $.
\end{lem}
\begin{proof}
We have
 \begin{align}
 & k \B(Q_{U_{h+1,d}}(t)-Q_{U_{h+1,d-1}}(t))+ \B(Q_{U_{h,d}}(t)-Q_{U_{h,d-1}}(t)) \notag \\
 &=\sum _{i=0}^{\floorfrac{d-1}{2} }
 \left( k \frac{((d - 2 i) (d + h +1- i) + i )(d+h)!}{( d + h - i) (d + h+1 - i) i!(d-i)!h! } +\right.\notag\\
 &\qquad\qquad\qquad \left. \frac{((d - 2 i) (d + h - i) + i )(d+h-1)!}{(-1 + d + h - i) (d + h - i) i!(d-i)!(h-1)! } \right) {\floorfrac{d-1}{2} \choose i } t^i \notag
 \\
 &= \frac{(d + h -1)!}{h! }\sum _{i=0}^{\floorfrac{d-1}{2} }
 { \frac{ p_3}{i! (d - i)! (d - i +h - 1) (d - i + h) (d- i + h + 1)} } {\floorfrac{d-1}{2} \choose i }t^i, \label{hh1}
 \end{align}
 where $p_3=k (d + h) (d + h - i - 1) ((d -2i)(d + h + 1 - i) + i ) +
 h (d + h - i + 1)\left((d - 2 i) (d + h - i) + i \right) .$
 Similar to the proof of Lemma \ref{m1}, we first prove that
 \[
 \sum _{i=0}^{\floorfrac{d-1}{2} }
 p_3 {\floorfrac{d-1}{2} \choose i }t^i
 \]
 has only real roots for $h\geq2$, which follows from the following lines.
 \begin{mma}
 \In p3= |Collect|\left[ \sum _ {i = 0}^
 n t^i\binom {n} {i} ( k (d + h) (d + h - i - 1) ((d -2i)(d + h + 1 - i) + i )\right. \linebreak +
 h (d + h - i + 1)\left.(d - 2 i) (d + h - i) + i )\right) / (t + 1)^{n - 3},
 t, \left. \right];\\
 \In disc = |Discriminant|[p3, t];\\
 \In |Resolve|[|ForAll|[\{k, n, h, d\}, n\geq3 \&\& h \geq 2 \&\& (d == 2n + 1\|d==2n+2), disc > 0]]\\
 \Out True\\
 \end{mma}
 For the case $h=1$,
 the right-hand side of \eqref{hh1} is
 $$d! \sum _{i=0}^{\floorfrac{d-1}{2} }
 { \frac{(d+1) k ((d-2 i) (d-i+2)+i)+(d-2 i+1) (d-i+2)}{i! (d- i + 2)!} } {\floorfrac{d-1}{2} \choose i }t^i .$$
 Its real-rootedness follows from the following lines.
 \begin{mma}
	\In p4= |Collect|\left[ \sum _ {i = 0}^
	n t^i\binom {n} {i} ( (d+1) k ((d-2 i) (d-i+2)+i)+(d-2 i+1) \right. \linebreak \cdot (d-i+2)
	 \left. \right) / (t + 1)^{n - 2},
	t, \left. \right];\\
	\In disc = |Discriminant|[p4, t];\\
	\In |Resolve|[|ForAll|[\{k, n, d\}, n\geq3 \&\& (d == 2n + 1\|d==2n+2), disc > 0]]\\
	\Out True\\
	\end{mma}
\noindent Thus, the roots of $\B(Q_{U_{h+1,d}}(t)-Q_{U_{h+1,d-1}}(t))$ and $\B(Q_{U_{h,d}}(t)-Q_{U_{h,d-1}}(t))$ are interlacing.
Similar to the proof of Lemma~\ref{m1}, to finish the proof, we need to show that the Wronskian of $\B(Q_{U_{h+1,d}}(t)-Q_{U_{h+1,d-1}}(t))$ and $\B(Q_{U_{h,d}}(t)-Q_{U_{h,d-1}}(t))$ is negative at $t=0$.
Recall that
	$$Q_{U_{m,d}}(t)=\binom {d + m - 1} {m}+ \frac {(d - 2) m} {d + m - 1}\binom {d + m} {m + 1} t+\cdots,$$
	and
	$$Q_{U_{m,d}}(t)-Q_{U_{m,d-1}}(t)=\binom {d + m - 2} {m - 1}+
	\frac {\left(d^2 + d (m - 3) - 2 m + 3 \right) (d + m - 3)!} {(d - 1)! (m - 1)!} t+\cdots.$$
	The constant term of $W[\B(Q_{U_{h+1,d}}(t)-Q_{U_{h+1,d-1}}(t)),\B(Q_{U_{h,d}}(t)-Q_{U_{h,d-1}}(t))]$ is
	\begin{align*}
	&\binom{\floorfrac{d-1}{2}}{1} \frac {\left(d^2 + d (h - 2) - 2 h+1\right) (d + h - 2)!} {(d - 1)! h!} \cdot \binom {d + h - 2} {h - 1} \\[6pt]
		&\qquad - \binom{\floorfrac{d-1}{2}}{1} \binom {d + h - 1} {h} \cdot \frac {\left(d^2 + d (h - 3) - 2 h + 3 \right) (d + h - 3)!} {(d - 1)! (h - 1)!} \\[6pt]
		&=- \binom{\floorfrac{d-1}{2}}{1} \frac{1}{d+h-1}{\binom{d+h-3}{h-1} \binom{d+h-1}{h}},
	\end{align*}
	which is negative.
	We complete the proof.
\end{proof}
\begin{lem}
For any positive integer $m\geq2$, we have
$$\B(Q_{U_{m,d}}(t)) \ll \B(Q_{U_{m,d}}(t)-Q_{U_{m,d-1}}(t)).$$
\end{lem}
\begin{proof}
 Through a direct computation, we obtain
 \begin{align*}
 &\B(Q_{U_{m,d}}(t)(t))+ k \B(Q_{U_{m,d}}(t)-Q_{U_{m,d-1}}(t)) \\
 &=\B( (k+1) Q_{U_{m,d}}(t)(t)-k Q_{U_{m,d-1}}(t)) \\
 &= \frac{(m-1+d)!}{(m-1)! }\sum _{i=0}^{\floorfrac{d-1}{2} }
 { \frac{ p_5}{i! (i + m) (d - i)! ( d - i + m-1) (d - i + m) } } {\floorfrac{d-1}{2} \choose i }t^i
 \end{align*}
 where $p_3=(k + 1) (d - 2 i) (d + m) (d - i + m - 1) - k (d - 2 i - 1) (d - i) (d - i + m).$
 Similar to the proof of Lemma, we first prove that
 \[
 \sum _{i=0}^{\floorfrac{d-1}{2} }
 p_5 {\floorfrac{d-1}{2} \choose i }t^i
 \]
 has only real roots, which follows from the following lines.
 \begin{mma}
 \In p5= |Collect|\left[ \sum _ {i = 0}^
 n t^i\binom {n} {i} ( (k + 1) (d - 2 i) (d + m) (d - i + m - 1) - k (d - 2 i - 1) \right. \linebreak (d - i) (d - i + m) )/ (t + 1)^{n - 3},
 t \left. \right];\\
 \In disc = |Discriminant|[p5, t];\\
 \In |Resolve|[|ForAll|[\{k, n, m, d\}, n\geq3 \&\& m \geq 1 \&\& (d == 2n + 1\|d==2n+2), disc > 0]]\\
 \Out True\\
 \end{mma}
 Recall that
 \begin{align*}
 Q_{U_{m,d}}(t) & =\binom {d + m - 1} {m}+ \frac {(d - 2) m} {d + m - 1}\binom {d + m} {m + 1} t+\cdots.\\
 Q_{U_{m,d}}(t)-Q_{U_{m,d-1}}(t)&=\binom {d + m - 2} {m - 1}+
 \frac {\left(d^2 + d (m - 3) - 2 m + 3 \right) (d + m - 3)!} {(d - 1)! (m - 1)!} t+\cdots.
\end{align*}
 The constant term of $W[\B(Q_{U_{m,d}}(t)),\B(Q_{U_{m,d}}(t)-Q_{U_{m,d-1}}(t))]$ is
 \begin{align*}
	 &\binom{\floorfrac{d-1}{2}}{1} \frac {(d - 2) m} {d + m - 1}\binom {d + m} {m + 1} \cdot \binom {d + m - 2} {m - 1} \\
	 &\qquad - \binom{\floorfrac{d-1}{2}}{1} \binom {d + m - 1} {m} \cdot \frac {\left(d^2 + d (m - 3) - 2 m + 3 \right) (d + m - 3)!} {(d - 1)! (m - 1)!} \\
	 &=-\frac {\left( m^2 +(d - 1) (d + 2 m - 2) + 1 \right)} {(d + m - 1) (d + m)} \binom {d+m - 3} {m - 1}\binom {d + m} {m + 1},
 \end{align*}
 which is negative.
\end{proof}
By combining the above three lemmas and Lemma~\ref{WeakTransitivity}, we prove Lemma~\ref{h_m_int}.
\begin{proof}[Proof of Lemma \ref{h_m_int}]
Let $f_{0}=\B(Q_{U_{m,d}}(t))$ and \vspace{-2mm}
$$f_h=\B(Q_{U_{m+1-h,d}}(t)-Q_{U_{m+1-h,d-1}}(t)) \text{\quad for\quad} 1\leq h\leq m.$$
It follows from Lemmas \ref{h_m_int}, \ref{m1}, and \ref{h+1h} that
$f_0 \ll f_{m}$ and $f_{i-1} \ll f_i$ for all $1 \leq i \leq m$. Thus the polynomial sequence $(f_0, f_1, \ldots, f_m)$ is interlacing by Lemma \ref{WeakTransitivity}.
In particular, we have $f_0 \ll f_h$ for any $ 1 \leq h \leq m$, i.e.,
$$\B(Q_{U_{m,d}}(t)) \ll \B (Q_{U_{h,d}}(t)-Q_{U_{h,d-1}}(t)) \text{\quad for \quad} 1\leq h\leq m.$$
We complete the proof.
\end{proof}
Although we do not know the specific range of $\lambda_h$, this does not affect us to prove Theorem \ref{paving-rz} and \ref{paving-log}.
\begin{proof}[Proof of Theorem \ref{paving-rz} and \ref{paving-log} ]
By Lemma~\ref{lem:convex},
we have $$\B(Q_{U_{m,d}}(t)) \ll \sum_{h=1}^{m} \lambda_h\cdot \B(Q_{U_{h,d}}(t)-Q_{U_{h,d-1}}(t))$$
for any nonnegative real numbers $\lambda_h$.
Thus, the polynomial $$\B(Q_{U_{m,d}}(t))-\sum_{h=1}^{m} \lambda_h\cdot \B(Q_{U_{h,d}}(t)-Q_{U_{h,d-1}}(t))$$ has only real roots for any nonnegative real numbers $\lambda_h$.
Suppose that $M$ is a paving matroid.
By \cite[Theorem 1.5]{braden2020singular}, we know that all the coefficients of $Q_M(t)$ are nonnegative.
By Theorem \ref{ch-n}, we conclude that the coefficients of $Q_M(t)$ forms a $\floorfrac{d-1}{2} $-sequence.
By definition, the degree of $Q_M(t)$ is less than or equal to $\floorfrac{d-1}{2}$. Again, by Theorem \ref{ch-n}, we know that
$\B(Q_M(t))$ has only real roots. This completes the proof of Theorem \ref{paving-rz}.
Theorem \ref{paving-log} follows from the well-known Newton's inequalities.
\end{proof}
It should be noted that although the degree of $Q_{U_{m,d}}(t)$ and $Q_{U_{h,d}}(t)-Q_{U_{h,d-1}}(t)$ for any $1\le d \le m$ is~${\lfloor (d-1) /2\rfloor}$,  the degree of $\B(Q_M(t))$ may be less than~${\lfloor (d-1) /2\rfloor}$.
In this sense,  the polynomial $\B(Q_M(t))$ may not equal $$\B(Q_{U_{m,d}}(t))-\sum_{h=1}^{m} \lambda_h\cdot \B(Q_{U_{h,d}}(t)-Q_{U_{h,d-1}}(t)).$$
This is why we adopt an indirect argument in the second paragraph of our proof of Theorem \ref{paving-rz}.
\section{Concluding remarks}
Based on the log-concavity properties established in this paper,  we shall study the degrees of the inverse Kazhdan-Lusztig polynomials and propose several intriguing conjectures.
 A matroid $M$ is said to be \emph{non-degenerate} if its Kazhdan-Lusztig polynomial has degree ${\lfloor(\mbox{rk}(M)-1)/2\rfloor}$, the maximal possible degree, see \cite{gedeon2017survey, vecchi2021matroid}.
Gedeon, Proudfoot, and Young~\cite{gedeon2017survey} conjectured that there are interlacing properties for the roots of the Kazhdan-Lusztig polynomial of a non-degenerate matroid $M$ and those of a contraction $M/{e}$ by one element.
However, similar properties do not hold for $\B(Q_M(t))$.
Nevertheless, one can  still investigate the relationship between the degrees of $P_M(t)$ and $Q_M(t)$.
Vecchi~\cite{vecchi2021matroid} obtained the following result for a matroid of odd rank and pointed out that this is not true when $M$ is of even rank.
 \begin{thm}[\cite{vecchi2021matroid}]\label{odd}
	Let $M$ be a matroid of odd rank $r$. Then $$\left[t^{\left\lfloor \frac{r-1}{2}\right\rfloor}\right]P_M(t) = \left[t^{\left\lfloor \frac{r-1}{2}\right\rfloor}\right]Q_M(t).$$
	In particular, a matroid $M$ of odd rank $r$ is non-degenerate if and only if $Q_M(t)$ has degree $\left\lfloor\frac{r-1}{2}\right\rfloor$.
\end{thm}
Theorem \ref{odd} inspires us to propose the following related conjecture.
\begin{conj}\label{same_deg}
For any matroid $M$, the degrees of $P_M(t)$ and $Q_M(t)$ are the same.
\end{conj}
We verified Conjecture~\ref{same_deg} for all matroids on at most 9 elements, as well as for the cycle matroids of graphs with at most 9 vertices.
Braden, Braden, Huh, Matherne, Proudfoot, and Wang~\cite{braden2020singular} proved that the coefficients of both $P_M(t)$ and $Q_M(t)$ are nonnegative.
But are the coefficients strictly positive?  We provide
  the following result.
\begin{prop}\label{paving_positive}
	For any paving matroid \( M \), each coefficient of \( Q_M(t) \) is positive, except possibly the coefficient of \( t^{\lfloor \frac{\text{rk}(M)-1}{2} \rfloor} \), which may be zero.
 \end{prop}
\begin{proof}
The linear combination of two real-rooted polynomials $f(t)$ and $g(t)$, each of degree $n$, is given by $a f(t) + b g(t)$. The degree of this combination is either $n$ or $n-1$, provided that $a^2 + b^2 \neq 0$.
Then, for any paving matroid $M$, the degree of $\B(Q_M(t)) $ is at least ${\lfloor(rk(M)-1)/2-1\rfloor}$ by the proof of Theorem \ref{paving-rz}. Thus, the coefficient of $t^{\lfloor(rk(M)-1)/2\rfloor-1}$ is positive. On the other hand, the constant term of $Q_M(t)$ is just the absolute value of the constant term of the characteristic polynomial of $M$, which is positive, see \cite[Proposition 3.1]{vecchi2021matroid} and \cite[Theorem 7.1.8]{zaslavsky1987mobius}.
It is well known that if the coefficients of a univariate polynomial are non-negative, log-concave, and have no internal zeros, then they are also unimodal. 
Thus, for any paving matroid $M$, each coefficient of the Kazhdan-Lusztig polynomial $Q_M(t)$ corresponding to $t^i$ with $i < \left\lfloor \frac{\mathrm{rk}(M) - 1}{2} \right\rfloor$ is positive.
\end{proof}
Based on Conjecture~\ref{same_deg}, we believe that an analog of Proposition~\ref{paving_positive} holds for the Kazhdan–Lusztig polynomials $P_M(t)$ as well. That is, for any paving matroid $M$, each coefficient of $P_M(t)$ is positive, except possibly the coefficient of $t^{\lfloor \frac{\mathrm{rk}(M)-1}{2} \rfloor}$.
 \section*{Acknowledgements}
 This work was supported by the National Natural Science Foundation of China.

\end{document}